\documentclass[11pt]{amsart}

\usepackage{amssymb,amsfonts,amsmath}
\usepackage{amsthm, stmaryrd}
\usepackage{graphicx}
\usepackage[all,arc]{xy}
\usepackage{hyperref}
\hypersetup{colorlinks,allcolors=violet}
\usepackage{enumerate}
\usepackage{bbm}
\usepackage{dsfont}
\usepackage{mathrsfs}
\usepackage{tikz-cd}
\usetikzlibrary{shapes}
\usepackage{import}
\usepackage{float}
\restylefloat{table}
\usepackage{multirow}
\usepackage{pdflscape}
\usepackage{listofitems}
\usepackage{standalone}

\usepackage[margin=1in]{geometry}
\usepackage{mathtools}
\usetikzlibrary{decorations.pathmorphing}
\usepackage{mathrsfs}
\newtheorem{thm}{Theorem}[section]
\newtheorem*{thm*}{Theorem}

\newtheorem{innercustomthm}{Theorem}
\newenvironment{customthm}[1]
  {\renewcommand\theinnercustomthm{#1}\innercustomthm}
  {\endinnercustomthm}
  
\newtheorem{innercustomcor}{Corollary}
\newenvironment{customcor}[1]
  {\renewcommand\theinnercustomcor{#1}\innercustomcor}
  {\endinnercustomcor}

\newtheorem{lem}[thm]{Lemma}

\theoremstyle{definition}
\newtheorem{defn}[thm]{Definition}

\newtheorem{rem}[thm]{Remark}

\newtheorem*{thm1.2}{\textrm{Theorem 1.2}}

\theoremstyle{remark}
\newcommand{\Mbar}{\overline{\mathcal{M}}}
\newcommand{\M}{\mathcal{M}}

\newtheorem{sch}[thm]{Scholium}

\newcommand{\Z}{\mathbb{Z}}

\newcommand{\QQ}{\mathbb{Q}}

\renewcommand{\P}{\mathbb{P}}
\newcommand{\bbS}{\mathbb{S}}

\newcommand{\Aut}{\operatorname{Aut}}

\usepackage{ytableau}

\newcommand{\val}{\operatorname{val}}

\newcommand{\ch}{\operatorname{ch}}
\newcommand{\bT}{\mathbf{T}}

\def\C{\mathbb{C}}

\def\F{\mathbb{F}}

\def\P{\mathbb{P}}

\def\T{\bT}

\def\Z{\mathbb{Z}}

\def\calB{\mathcal{B}}
\def\calC{\mathcal{C}}

\def\calM{\mathcal{M}}

\def\calR{\mathcal{R}}

\def\calW{\mathcal{W}}
\def\calX{\mathcal{X}}
\def\calY{\mathcal{Y}}
\def\calZ{\mathcal{Z}}
\def\M{\mathcal{M}}

\newcommand{\Gr}{\mathrm{Gr}}

\newcommand{\Ind}{\operatorname{Ind}}

\newcommand{\Spec}{\operatorname{Spec}}

\newcommand{\Conf}{\mathrm{Conf}}

\newcommand\cycle[2][\,]{%
  \readlist\thecycle{#2}%
  (\foreachitem\i\in\thecycle{\ifnum\icnt=1\else#1\fi\i})%
}

\newcommand{\CH}{\operatorname{CH}}
\makeatletter
\let\c@equation\c@thm
\makeatother
\numberwithin{equation}{section}
\usepackage{makecell}
 \usepackage[style=alphabetic,doi=false,isbn=false,url=false,eprint=true,backend=biber,giveninits=false,maxalphanames=5,maxcitenames=3,maxbibnames=8,giveninits=true,clearlang=true,hyperref]{biblatex}

\AtBeginBibliography{\footnotesize}
\AtEveryBibitem{\clearlist{language}}
\graphicspath{ {images/} }
\title{The $\bbS_n$-equivariant Chow polynomial of the Braid matroid}

\author[S. Kannan]{Siddarth Kannan}
\address{Department of Mathematics, Massachusetts Institute of Technology, Cambridge, MA, USA}
\email{\url{spkannan@mit.edu}}

\author[L. K\"uhne]{Lukas K\"uhne}
\address{Fakult\"at f\"ur Mathematik, Universit\"at Bielefeld, Bielefeld, Germany}
\email{\url{lkuehne@math.uni-bielefeld.de}}

\begin{document}
\maketitle\thispagestyle{empty}

\begin{abstract}
We determine the generating function for the $\bbS_n$-equivariant Chow polynomials of the braid matroid $B_n$. The Chow polynomial of $B_n$ is the Poincar\'e polynomial of the wonderful compactification of the complement of the braid arrangement, with respect to the maximal building set. A key input to our result is the identification of this wonderful compactification with a moduli space of multiscale differentials, recently established by Devkota, Robotis, and Zahariuc. In this way, our work also contributes to the literature on topology of moduli spaces of multiscale differentials. To prove our main formula, we study the classes of these moduli spaces in the Grothendieck ring of varieties via the formalism of $\bbS$-spaces developed by Getzler and Pandharipande. We also give a new interpretation of the numerical Chow polynomial of $B_n$ as the Poincar\'e polynomial of a moduli space of genus-zero relative stable maps to $\P^1$.
\end{abstract}

\section{Introduction}
For $n \geq 2$, the braid matroid $B_n$ is determined by the normal vectors of the braid hyperplane arrangement
\begin{equation}\label{eqn:braidarrangement}
	\{z_i = z_j \mid 1 \leq i< j \leq n\} \subseteq \C^n/\C\cdot (1, \ldots, 1).
\end{equation}
 It is also the graphic matroid of the complete graph $K_n$. The rational Chow ring\footnote{All Chow rings and Chow groups in this paper are taken with $\QQ$-coefficients.} $\CH^\star(B_n)$ with respect to the maximal building set is of interest in matroid and lattice theory. It coincides with the rational Chow ring of the wonderful compactification $Y_n$ of the complement of the braid arrangement, in the sense of De Concini--Procesi \cite{concini}.

 The symmetric group $\bbS_n$ acts on $B_n$, and this makes $\CH^\star(B_n)$ into a graded $\bbS_n$-representation. In this paper we are interested in the $\bbS_n$-equivariant Chow polynomial $\mathrm{H}_n(t)$ of $B_n$, which is defined as the $\bbS_n$-equivariant Hilbert--Poincar\'e series of the Chow ring:
\[
	\mathrm{H}_n(t) := \sum_{k \geq 0} \ch_n( \CH^{k}(B_n)) \cdot t^{k} \in \Lambda[t].
\]
Above, $\Lambda := \QQ[\![p_1, p_2, \ldots]\!]$ is the degree-completed ring of symmetric functions over $\QQ$ (see \S \ref{subsec:sym} for a precise definition), and for an $\bbS_n$-representation $V$ we write $\ch_n(V) \in \Lambda$ for its Frobenius characteristic. The element $\ch_n(V)$ determines the decomposition of $V$ into irreducible $\bbS_n$-representations.

Assemble the Chow polynomials into the generating function
\[ \mathsf{B} := \sum_{n \geq 2} \mathrm{H}_n(t) \in \Lambda [\![ t]\!]. \]
Let
\[ \M_{0, n+1} := \frac{\Conf_{n + 1}(\P^1)}{\Aut(\P^1)} \]
denote the moduli space of $n+1$ ordered distinct points on $\P^1$, and set
\[ M_n(t) : = \sum_{i \geq 0} (-1)^i \ch_n(H^i_c(\M_{0, n+1};\QQ)) t^{i - n + 2}. \]
Up to signs and a degree shift, $M_n(t)$ is the $\bbS_n$-equivariant compactly-supported Poincar\'e polynomial of $\M_{0, n+1}$; Getzler \cite[Theorem 5.7]{GetzlerGenusZero} derived an explicit closed formula for the generating function
\[ \mathsf{M} := \sum_{n \geq 2} M_n(t) \in \Lambda [\![ t ]\!]. \]
Our main theorem determines $\mathsf{B}$ in terms of $\mathsf{M}$. For $n \geq 1$, let $h_n \in \Lambda$ denote the $n$th homogeneous symmetric function; these are characterized by the formula
\[ \sum_{n\geq0}h_nt^n = \exp\left( \sum_{n \geq 1} \frac{p_n}{n}t^n\right), \]
so e.g. $h_1 = p_1$ and $h_2 = (p_1^2 + p_2)/2$.
\begin{customthm}{A}\label{thm:chow_gen_fun}
Let $\circ$ denote plethysm of symmetric functions. Then the generating function $\mathsf{B}$ is uniquely determined by the formula
\[  (h_1 + \mathsf{B}) \circ (h_1 + (t - 1)\mathsf{M}) =h_1 + t \cdot \mathsf{B}, \]
together with the initial condition $\mathrm{H}_2(t) = h_2$. 
\end{customthm}

See Table \ref{table:frob_chars} for the values of $\mathrm{H}_n(t)$ for $n \leq 6$.  Our proof of Theorem \ref{thm:chow_gen_fun} relies on a modular interpretation of the wonderful compactification $Y_n$ as a moduli space $\calB_n$ of \textit{multiscale differentials} in the sense of \cite{multiscale}. The identification we use has recently been established by Devkota, Robotis, and Zahariuc \cite{DRZ}; see \S\ref{subsec:ag_context} for a more in-depth discussion of the geometric context.

We can extract numerical results from Theorem \ref{thm:chow_gen_fun} as follows: set $\mathrm{H}^{\mathrm{num}}_1(t) =1$, and for $n \geq 2$ set
\[
	\mathrm{H}^{\mathrm{num}}_n(t) := \sum_{k \geq 0} \dim_{\QQ} \CH^{k}(B_n) \cdot t^k.
\]
This is the usual Chow polynomial of the braid matroid, which has attracted significant attention in matroid theory; see \S\ref{sec:matroids} below.

\begin{customcor}{B}\label{cor:numerics}
Let $s(n, k)$ denote the signed Stirling number of the first kind, and let $S(n, k)$ denote the Stirling number of the second kind. We have the following recursions: 
\begin{enumerate}
    \item
    For $n\ge 2$, we have
    \[ (t - 1) \cdot \mathrm{H}_{n}^{\mathrm{num}}(t) = \sum_{k = 1}^{n - 1} \mathrm{H}_{k}^{\mathrm{num}}(t) \cdot \sum_{j = k}^{n} s(n, j) S(j, k) t^{j - k}. \]
    \item Set $\chi_1 = 1$ and for $n\geq 2$ set $\chi_n := \dim_{\QQ} \CH^\star(B_n)$. Then
    \[ \chi_n= \sum_{k = 1}^{n -1} \chi_k \cdot \binom{n}{k - 1} \cdot (n - k - 1)! \cdot (-1)^{n - k - 1}. \]
\end{enumerate}
\end{customcor}
While Theorem \ref{thm:chow_gen_fun} is new, the statement of Corollary \ref{cor:numerics} can also be derived from known general recursions for Chow polynomials; see Remark \ref{rem:general_recursion}. An alternative approach to the calculation of $\mathrm{H}^{\mathrm{num}}_n(t)$ was given by Gaiffi--Serventi\footnote{Their formula is stated in terms of the Hilbert--Poincar\'e series of the rational cohomology ring. For wonderful compactifications of hyperplane arrangement complements, this series equals the Chow polynomial, under their grading convention.} (\cite[Theorem 3.1]{GS}); it is not clear to us how to extract the formulas of Corollary \ref{cor:numerics} directly from their work. 

\begin{table}[h]
\begin{tabular}{|l|l|}
\hline
$n$ & $\mathrm{H}_n(t)$                                                                                    \\ \hline
$2$ & $s_2$                                                                                       \\ \hline
$3$ & $s_3(1 + t)$                                                                              \\ \hline
$4$ & $s_4(1 + 3t + t^2) + s_{31}t+ s_{22}t $                                              \\ \hline
$5$ & $s_5(1 + 5t + 5t^2 + t^3) + s_{41}(4t + 4t^2) + s_{32}(3t + 3t^2) + s_{221}(t + t^2)$ \\ \hline
$6$ & \makecell{$s_6(1+9t+19t^2+9t^3+t^4) + s_{51}(7t+21t^2+7t^3)+ s_{42}(9t+28t^2+9t^3) +$ \\$s_{411}(t+7t^2+t^3)+ s_{33}(2t+8t^2+2t^3)+ s_{321}(2t+12t^2+2t^3)+ $ \\ $s_{3111} t^2  + s_{222}(2t+7t^2+2t^3) +s_{2211} t^2$}                                          \\ \hline
\end{tabular}
\caption{The $\bbS_n$-equivariant Chow polynomial of $B_n$ for $n \leq 6$. For a partition $\lambda \vdash n$, we write $s_\lambda \in \Lambda$ for the corresponding Schur function.}
\label{table:frob_chars}
\end{table}

In addition to the above formulas, we give a new interpretation of the Chow polynomial as the Hilbert--Poincar\'e series of the Chow ring of a certain moduli space of relative stable maps to $\P^1$.
\begin{customthm}{C}[Theorem \ref{thm:relative_stable_maps}]\label{thm:rel_maps_intro}
    If $\calR_n$ denotes the moduli space of genus-zero relative stable maps to~$\P^1$ as in Definition \ref{defn:Rn}, then for all $k \geq 0$ there is an equality of dimensions
    \[\dim_{\QQ} \CH^{k}(\calR_n) = \dim_{\QQ}\CH^k(B_n).\]
\end{customthm}
The Chow ring of $\calR_n$ encodes the relative Gromov--Witten theory of $\P^1$ with maximal contact in the sense of \cite{NR}. Theorem \ref{thm:rel_maps_intro} thus relates the combinatorics of the braid matroid with the enumerative geometry of $\P^1$. In particular, the formulas in Corollary \ref{cor:numerics} apply to the moduli space~$\calR_n$, whose topology was previously studied by the first author in \cite{KannanP1}. 

\subsection{Method} 
We prove Theorem~\ref{thm:chow_gen_fun} by analyzing the classes of the moduli spaces $\calB_n$ of multiscale differentials in the Grothendieck ring of varieties. To consider all relevant moduli spaces at once, we use the formalism of $\bbS$-spaces developed by Getzler and Pandharipande \cite{GetzlerPandharipande}. In particular, the main technical Theorem \ref{thm:key_identity} upgrades Theorem \ref{thm:chow_gen_fun} in that it completely determines the class of $\calB_n$ in the Grothendieck group of $\bbS_n$-varieties over $\C$. The proof of this theorem uses the enumerative combinatorics of genus-zero \textit{level graphs}, which encode strata in the moduli spaces of multiscale differentials. Since $\calB_n$ is smooth, proper, and has the so-called \textit{Chow--K\"unneth generation property} as in \cite{CL}, the formula for the ranks of the Chow groups follows from our formula in the Grothendieck group.

To prove Theorem \ref{thm:rel_maps_intro}, we prove that the classes of the moduli space $\calR_n$ and $\calB_n$ coincide in the Grothendieck ring of varieties, using the stratification of $\calR_n$ studied in \cite{KannanP1}. The conclusion follows because $\calR_n$ also has the Chow--K\"unneth generation property. The equality of Chow polynomials is interesting because the spaces $\calB_n$ and $\calR_n$ are not isomorphic: the former is a smooth projective variety, while the latter is only smooth as a Deligne--Mumford stack.

\subsection{Algebro-geometric context}\label{subsec:ag_context}
When the minimal building set is used for the braid matroid, the corresponding wonderful compactification coincides with the Deligne--Mumford--Knudsen moduli space $\Mbar_{0, n+1}$ of stable genus-zero curves. The Chow rings of these spaces were first computed by Keel \cite{Keel}, who also gave a recursion determining their Betti numbers. A different approach to the Betti numbers using generating functions and tree sums was given by Manin \cite{ManinTrees}, and later made $\bbS_n$-equivariant by Getzler via the language of cyclic operads \cite{GetzlerGenusZero}. 

Getzler's calculations can also be phrased in the language of $\bbS$-spaces, developed later by Getzler and Pandharipande en route to their calculation of the Betti numbers of the Kontsevich moduli space $\Mbar_{0, n}(\P^r, d)$ of genus-zero stable maps to projective space \cite{GetzlerPandharipande}. Their result also relies on enumerative combinatorics: they express the relevant generating function as an infinite sum over trees, and use a composition operation on $\bbS$-spaces to solve the tree sum recursively. 

As in the case of genus-zero stable curves and maps, our work relies on the recursive behavior of an infinite sum over trees, with the added twist of level structures which reflect the geometry of multiscale differentials. The geometry and topology of mulitscale differentials and their moduli has attracted significant interest in recent years; see e.g. \cite{multiscale, ChernMultiscale, MultiscaleConnected, TaleOfTwoModuliSpaces, ChenLarson, TostesonAbelian, ChenTaut, nontrivalnontaut} and the references theirein. In particular, Devkota calculated the Chow ring of $\calB_n$ \cite{Devkota}, and his presentation also follows from the identification of the Chow ring of $\calB_n$ with that of the braid matroid $B_n$ \cite{DRZ}. In this way, our work is a new contribution both to the topology of moduli of multiscale differentials and to the combinatorics of matroids.

In both our work and in the minimal building set case of $\Mbar_{0, n+ 1}$, the tree sum approach yields a functional equation which determines the generating function for Poincar\'e polynomials, in that one can algorithmically compute the coefficients recursively. Aluffi, Marcolli, and Nascimento \cite{AluMarNas} used the relevant functional equation to prove remarkably explicit closed formulas for the Grothendieck ring class $[\Mbar_{0, n+1}]$ which eliminate the need for the recursive calculation. We would be interested to know whether similar techniques can produce closed formulas for the Grothendieck ring class of $\calB_n$, starting from the recursion developed herein, and to what extent their techniques can be made $\bbS_n$-equivariant.

\subsection{Matroid-theoretic context}\label{sec:matroids}
Feichtner and Yuzvinsky described the Chow ring of the De Concini--Procesi compactification of a hyperplane arrangement complement in terms of the underlying matroid~\cite{FeichtnerYuzvinsky}, which naturally led to the notion of the Chow ring $\CH^\star(M)$ of an arbitrary matroid $M$. This ring depends on the choice of building set, and we employ the usual convention that $\CH^\star(M)$ is defined with respect to the maximal building set. For a matroid $M$, we use the notation $\mathrm{H}^{\mathrm{num}}_M(t)$ for the Chow polynomial, which is the Hilbert--Poincar\'e series of $\CH^\star(M)$. The study of the Chow polynomial has recently attracted significant interest. It is for instance conjectured that $\mathrm{H}_M^{\mathrm{num}}(t)$ has only real roots~\cite{FS24}; this was recently confirmed for all uniform matroids~\cite{BV25}, and for a larger class of matroids which includes the braid matroid \cite{UMEL}.

Adiprasito, Huh and Katz established a Hodge theory for matroids by proving that $\CH^\star(M)$ satisfies the Kähler package \cite{AHK}. Their work implies that $\mathrm{H}_M^{\mathrm{num}}(t)$ is a monic and palindromic polynomial with unimodal coefficients. Ferroni, Matherne, Stevens and Vecchi proved that the Chow polynomial is the inverse of the reduced characteristic polynomial in the incidence algebra of the lattice of flats~\cite{FMSV24}.
This leads to a general recursion to compute the Chow polynomial of a matroid, previously obtained by Jensen--Kutler--Usatine \cite{JKUmotivic}. Their work also allowed for an even more general definition of Chow polynomials for posets~\cite{FMV24}.
Recently, Stump produced a different formula for the Chow polynomial of a matroid using the descents along the chains in the lattice~\cite{Stu24}.
Stump's formula implies that the Chow polynomial is $\gamma$-positive. The $\gamma$-positivity of the Chow polynomial is also proved in \cite[Theorem 3.25]{FMSV24}.

In ~\cite{EFMPV25}, Eur et al. give a recursive formula for the Hilbert--Poincar\'e series of the Chow ring of a matroid is given for an arbitrary building set. In the case of the minimal building set for $B_n$, they use their recursion to re-prove the formula, discussed in \S\ref{subsec:ag_context} above, for the Poincar\'e polynomial of the moduli space $\Mbar_{0, n + 1}$ given by Aluffi, Marcolli, and Nascimento~\cite{AMM24}.

When it comes to the $\bbS_n$-action, Angarone--Nathanson--Reiner \cite{ANRperm} prove that for any matroid $M$, each Chow group $\mathrm{CH}^k(M)$ is a permutation representation of $\Aut(M)$. Their result implies that the $\bbS_n$-representations determined by Theorem \ref{thm:chow_gen_fun} are all in fact permutation representations of $\bbS_n$. Stembridge \cite{Stembridge} computed the character of $\bbS_n$ acting on the Chow ring of the Boolean matroid, and Liao gave an explicit permutation basis \cite{Liao2}. In this case the action of $\bbS_n$ is a restriction of an action of $\bbS_2 \times \bbS_n$, whose character was determined by Bergstr\"om--Minabe \cite{BergstromMinabe2}. Liao \cite{Liao} has also determined the action of $\bbS_n$ on (augmented) Chow rings of uniform matroids.
In this context, Theorem~\ref{thm:chow_gen_fun} presents progress on Question 5.1 in~\cite{ANRperm} which asks about explicit expressions for other families of matroids with symmetry.

\subsection{Outline} In \S\ref{section:prelim} we recall the moduli space $\calB_n$ of projectivized multiscale differentials as studied in \cite{DRZ}, with a particular focus on the stratification of the space by so-called level trees. In \S\ref{sec:sym} we recall $\bbS$-spaces and their basic properties established by Getzler--Pandharipande. We then reduce Theorem \ref{thm:chow_gen_fun} to a combinatorial identity in the Grothendieck ring of $\bbS$-spaces. We prove this identity in \S\ref{sec:motivic_proof}, using the recursive structure of the level tree stratification. We derive Corollary \ref{cor:numerics} in \S\ref{sec:numerical} using basic properties of symmetric functions. Finally, we prove Theorem \ref{thm:rel_maps_intro} in \S\ref{sec:relmaps}.

\subsection*{Acknowledgments}
We would like to thank Ana Mar\'ia Botero, Sarah Brauner, Megan Chang-Lee, Matthew Litman, Navid Nabijou, Dhruv Ranganathan, Vic Reiner, and Terry Dekun Song for helpful discussions and correspondence. We are especially grateful to Robert Angarone for sharing sample calculations of $\mathrm{H}_n(t)$. SK is supported by an NSF Postdoctoral Fellowship (DMS-2401850). LK is supported by the Deutsche Forschungsgemeinschaft (DFG, German Research Foundation) -- SFB-TRR 358/1 2023 -- 491392403 and SPP 2458 -- 539866293.

\section{The moduli space of multiscale differentials}\label{section:prelim}
In this section we recall the definition of the moduli space
\begin{equation}\label{eqn:introducing_Bn}
    \calB_n = \P\Xi\Mbar_{0, n+1}(\underbrace{ 0,\ldots, 0}_{n}, -2)
\end{equation}
studied in \cite{Devkota,DRZ}. The moduli space $\calB_n$ parameterizes stable $(n+1)$-marked curves of genus zero, together with certain choices of meromorphic differentials on the components of the curve. The choices of differentials are constrained by the dual tree of the stable curve. As such, we begin by setting some notation for working with dual trees.
\subsection{Stable trees}
A \emph{tree} is a finite connected graph with no cycles. Given a tree $T$, we write $V(T)$ and $E(T)$ for its sets of vertices and edges, respectively. Given $v \in V(T)$, write $\val(v)$ for the \textit{valence} of $v$, which is equal to the number of edges containing $v$.
\begin{defn}
Let $X$ be a finite set. An \emph{$X$-marked stable tree} is a pair $\mathbf{T}= (T, m)$ where $T$ is a tree, and $m: X  \to V(T)$ is a function such that
\[ \val(v) + |m^{-1}(v)| \geq 3 \]
for all $v \in V(T)$. 
\end{defn}
We use the notation $V(\bT) := V(T)$ and $m_{\bT} := m$. For any vertex $v \in V(\bT)$ we write \[\mathrm{deg}(v):= \mathrm{val}(v) + |m_{\bT}^{-1}(v)|.\]
Write $\Gamma_{0, X}$ for the set of isomorphism classes of $X$-marked stable trees. 
\begin{defn}
When $X = \{0, 1, \ldots, n\}$, we set
\[\Gamma_{0, n}^{\star}:= \Gamma_{0, 
X},\] and refer to an element $\bT \in \Gamma_{0, n}^\star$ as a \textit{rooted $n$-pointed stable tree}, where the vertex \[v_0(\bT) := m_{\bT}(0) \in V(\bT)\]
is defined to be the \textit{root} of $\T$. 
\end{defn}
Note that each tree $\bT \in \Gamma_{0, n}^\star$ is naturally directed, away from the root. 
\begin{defn}
Given $\bT \in \Gamma_{0, n}^\star$, define the \textit{canonical partial order} $\leq_\bT$ on the vertex set $V(\bT)$, by declaring $v_1 \leq v_2$ if and only if $v_1 = v_2$ or there is a nonempty directed path from $v_1$ to $v_2$ in $\bT$. We write $v_1 <_\bT v_2$ if and only if there is a nonempty directed path from $v_1$ to $v_2$ in $\bT$.
\end{defn}
\subsection{The moduli space $\calB_n$}
Let $\Mbar_{0, n+1}$ denote the moduli space of $(n+1)$-pointed stable curves of genus zero, where the marked points are labelled by the set $\{0, \ldots, n\}$. Each genus-zero, $(n+1)$-pointed stable curve $(C, q_0, \ldots, q_n)$ naturally has a dual tree $\bT_{C} =(T_C, m_C) \in \Gamma_{0, n}^\star$; the vertex set $V(T_C)$ is the set of irreducible components of $C$, the edge set $E(T_C)$ is the set of nodes of $C$, and the marking function $m_C$ tracks the distribution of marked points among the irreducible components.
\begin{defn}
Let $(C, q_0, \ldots, q_n)$ be an $(n+1)$-pointed stable curve of genus zero, and let $C_v$ be an irreducible component of $C$, corresponding to a vertex $v \in V(\bT_C)$. Let $\eta$ be a node of $C$ which is contained in $C_v$. We say the node $\eta$ is \textit{incoming} at $C_v$ if the edge $e \in E(\bT_C)$ representing $\eta$ is directed towards $v$. If $v$ is a non-root vertex of $\bT_C$, then there is a unique node which is incoming at $C_v$, since $\bT_C$ is a connected tree.
\end{defn}

We can now describe the space $\calB_n$ from (\ref{eqn:introducing_Bn}). We will only describe the closed points of the moduli space, since foundational aspects are treated in \cite{multiscale}. The space $\calB_n$ is a fine moduli space for tuples \begin{equation}\label{eqn:diff_tuple}(C, q_0, \ldots, q_n; \ell, \{\omega_v\}_{v \in V(T_C)})
\end{equation}
with the following properties.
\begin{itemize}
\item $(C, q_0, \ldots, q_n)$ is an $(n+1)$-pointed stable curve of genus $0$ with marked dual tree $\bT_C \in \Gamma_{0, n}^\star$.
\item $\ell: V(\bT_C) \to \{0, \ldots, k\}$ is a surjective map for some $k \geq 0$, such that whenever $v <_T w$, we have $\ell(v) < \ell(w)$.
\item If $v \in V(\bT_C)$ is the root vertex of $\bT_C$ and $C_v$ denotes the corresponding irreducible component of $C$, then $\omega_v \in H^0(C, \Omega_{C_v}(2q_0))$ is a non-zero meromorphic differential on $C_v$ with a unique pole of order $2$ at~$q_0$.
\item If $v \in V(\bT_C)$ is a non-root vertex and $C_v \subset C$ is the corresponding irreducible component with incoming node $q_v \in C_v$, then $\omega_v \in H^0(C_v, \Omega_{C_v}(2q_v))$ is a meromorphic differential on $C_v$ with a unique pole of order $2$ at $q_v$.
\end{itemize} 

The function $\ell$ is called a \textit{level structure} on the dual graph $\bT_C$, and we call the union of components $C_v \subset C$ with $\ell(v) = j$ the \textit{$j$th level} of $C$. Each level $j$ of $C$ admits a diagonal $\C^\star$-action, which corresponds to simultaneously rescaling the differentials $\{\omega_{v} \mid \ell(v) = j \}$. If the number of levels is $e$, this leads to a $(\C^\star)^{e}$-action on such a tuple. Two tuples as in (\ref{eqn:diff_tuple}) are considered equivalent if they differ by the $(\C^\star)^e$-action. This torus is usually called the \textit{level rotation torus} in the literature on multiscale differentials.

\begin{defn}
    We refer to the tuples corresponding to closed points of $\calB_n$ as in (\ref{eqn:diff_tuple}) as \textit{multiscale differentials}. For such a tuple, the marked point $q_0$ is called the \textit{distinguished marking}, and the marked points $q_1, \ldots, q_n$ are the \textit{non-distinguished markings}. 
\end{defn}

In \cite{Devkota}, Devkota shows that the space $\calB_n$ is a smooth projective variety of dimension $n - 3$. In \cite[Theorem 3.1]{DRZ}, Devkota, Robotis, and Zahariuc identify the moduli space $\calB_n$ with the De Concini--Procesi compactification $Y_n$ of the complement of the braid arrangement (\ref{eqn:braidarrangement}) with respect to the maximal building set, by identifying both spaces with the same iterated blowup of $\Mbar_{0, n+ 1}$. Their identification immediately implies the following theorem.
\begin{thm}\label{thm:equiv_iso}
There is a $\bbS_n$-equivariant isomorphism of rational Chow rings
\[ \CH^\star(\calB_n) \cong \CH^\star(B_n). \]
\end{thm}

According to \cite{DRZ}, Theorem \ref{thm:equiv_iso} can also be deduced from the description of the log Chow ring of $\Mbar_{0, n+1}$ in \cite[Theorem 11]{PRSS}. Theorem \ref{thm:equiv_iso} lets us use the geometry of $\calB_n$ to calculate the $S_n$-equivariant Chow polynomial of $B_n$.

\subsection{Multiscale differentials on smooth curves} As established in \cite{multiscale}, the moduli space $\calB_n$ is a simple normal crossings compactification of the subspace
\[\calB^\circ_n \subset \calB_n\] where the underlying curve of the multiscale differential is smooth. Towards better understanding the geometry of $\calB_n$, let us briefly describe $\calB^\circ_n$. First, denote by
\[ \pi: \calC_{0, n+1} \to \calM_{0, n + 1} \]
the universal curve. Let $\omega_{\calC_{0, n+1}/\M_{0, n + 1}}$ denote the relative dualizing sheaf of this morphism. There are $n + 1$ sections \[\mathrm{s}_0, \ldots, \mathrm{s}_n: \M_{0, n+1} \to \calC_{0, n + 1} \]
of $\pi$ determined by the marked points, and we can thus define the twist $\omega_{\calC_{0, n+1}/\M_{0, n + 1}}(2\mathrm{s}_0)$. The pushforward
\[ \pi_*(\omega_{\calC_{0, n+1}/\M_{0, n + 1}}(2\mathrm{s}_0)) \]
is a line bundle on $\M_{0, n + 1}$, with fiber over $(C, q_0, \ldots, q_n)$ given by $H^0(C, \Omega_C(2q_0))$: the higher cohomologies of the sheaves $\Omega_C(2q_0)$ vanish, and $\dim_{\C} H^0(C, \Omega_C(2q_0)) = 1$ for all $C$. 

\begin{defn}\label{defn:Omega_star}
We write $\Omega \M_{0, n + 1}$ for the total space of the line bundle $\pi_*(\omega_{\calC_{0, n+1}/\M_{0, n + 1}}(2\mathrm{s}_0))$ over $\M_{0, n + 1}$. We also write $\Omega^\star\M_{0, n+1}$ for the complement of the zero section in $\Omega \M_{0, n + 1}$. 
\end{defn}
The variety $\Omega^\star\M_{0, n+1}$ parameterizes irreducible smooth pointed genus-zero curves $(C, q_0, \ldots, q_n)$ together with a nonzero section of $H^0(C, \Omega_C(2q_0))$. 

\begin{lem}\label{lem:bundlesoversmooth}
    There are $\bbS_n$-equivariant isomorphisms
    \[ \Omega^\star \M_{0, n + 1} \cong \C^\star \times \M_{0, n+ 1} \]
    and
    \[ \calB^\circ_n \cong \M_{0, n + 1}. \]
\end{lem}

\begin{proof}
    The Picard group $\mathrm{Pic}(\M_{0, n + 1})$ is trivial, and $\Omega^\star \M_{0, n + 1}$ is the complement of the zero section in the line bundle $\Omega \M_{0, n+1}$. Therefore $\Omega^\star \M_{0, n+1} \to \M_{0, n + 1}$ is a trivial $\C^\star$-bundle, which is $\bbS_n$-equivariant by construction. The second isomorphism follows from the observation that
    \[ \calB^\circ_n \cong \Omega^\star \M_{0, n + 1} /\C^\star \]
    where $\C^\star$ acts by scaling the differential.
\end{proof}

We will now describe the strata in the dual tree stratification of $\calB_n$. The geometry of each stratum is concretely understood in terms of the trivial torus bundles $\Omega^\star \M_{0, n + 1}$.
\subsection{The stratification by level trees}
Each point of the moduli space $\calB_n$ determines a dual tree, with the additional data of a level structure.
\begin{defn}
An \textit{$(n+1)$-pointed stable level tree} is a pair $(\bT, \ell)$ where $\bT \in \Gamma_{0, n}^\star$ and \[\ell: V(\bT) \to \{0, \ldots,k \}\] is a level structure for some integer $k \geq 0$. We call the integer $k + 1$ the \textit{length} of the level tree. We write $\Gamma_{0, n}^{\mathrm{lev, \star}}$ for the set of isomorphism classes of $(n+1)$-pointed stable level trees, and \[\Gamma_{0, n}^{\mathrm{lev, \star}}(k) \subset \Gamma_{0, n}^{\mathrm{lev, \star}}\] for the subset of level trees of length $k + 1$.
\end{defn}
We use the set $\Gamma_{0, n}^{\mathrm{lev, \star}}$ to define a stratification of the moduli space $\calB_n$. Given $(\bT, \ell) \in \Gamma_{0, n}^{\mathrm{lev, \star}}$, we write
\[\calB_{(\bT, \ell)} \subset \calB_{n}\]
for the locally closed stratum where the level tree underlying the stable curve is equal to $(\bT, \ell)$. The following lemma is immediate from the definition of a multiscale differential.
\begin{lem}\label{lem:stratum_factorization}
Suppose that $(\bT, \ell) \in \Gamma_{0, n}^{\mathrm{lev, \star}}(k)$. For each integer $j$ such that $j \in \{0, \ldots, k\}$, write
\[ V_j(\bT) := \ell^{-1}(j) \subseteq V(\bT).\]

Then there is an isomorphism
\[\calB_{(\bT, \ell)} \cong \prod_{j = 0}^{k} \left(\prod_{v \in V_j(\bT)} \Omega^\star \M_{0, \deg(v)} \right)/\C^\star, \]
where the $\C^\star$-action on $\prod_{v \in V_j(\bT)} \Omega^\star \M_{0, \deg(v)}$ is diagonal and free.
\end{lem}

\begin{rem}
The definition of $(n + 1)$-pointed stable level trees is equivalent, \textit{mutatis mutandi}, to Stanley's definition of \textit{total partitions} of $[n]$: see \cite[Example 5.2.5]{EC2}, especially Figure 5-3. This defines a straightforward $\bbS_n$-equivariant bijection between $\Gamma_{0, n}^{\mathrm{lev}, \star}$ and the cones of the \textit{Bergman fan} of the braid matroid, see \cite[Figure 5]{ArdilaKlivans}. 
\end{rem}

\section{Symmetric functions and $\bbS$-spaces}\label{sec:sym}

In this section we set up a motivic identity (Theorem \ref{thm:key_identity}), and explain how Theorem \ref{thm:chow_gen_fun} follows from it. The proof of the identity will be given in \S\ref{sec:motivic_proof}. The identity is stated in the language of Serre characteristics of $\bbS$-spaces as developed by Getzler--Pandharipande \cite{GetzlerPandharipande}, which we now recall.
\subsection{$\bbS$-spaces and associated operations} An $\bbS$-space $\calX$ is a sequence of $\C$-varieties $\calX(n)$ for each integer $n \geq 0$, such that $\calX(n)$ has an action of $\bbS_n$. Let $K_0(\mathsf{Var}; \bbS)$ denote the Grothendieck group of $\bbS$-spaces: this is the free abelian group generated by isomorphism classes of $\bbS$-spaces, with the relation
\[ [\calX] = [\calX \smallsetminus \calY] + [\calY] \] whenever $\calY \hookrightarrow \calX$ is an $\bbS$-subvariety. There is an isomorphism
\[ K_0(\mathsf{Var};\bbS) \cong \prod_{n \geq 0} K_0(\mathsf{Var};\bbS_n)\]
where $K_0(\mathsf{Var};\bbS_n)$ is the Grothendieck group of $\bbS_n$-varieties. The group $K_0(\mathsf{Var};\bbS)$ is made into a ring using the box product:
\[ [\calX] \cdot [\calY] = [\calX \boxtimes \calY], \]
where $\calX \boxtimes \calY$ is defined by
\begin{equation}\label{eqn:boxtimes}
    (\calX \boxtimes \calY)(n) = \coprod_{k = 0}^{n} \Ind_{\bbS_k \times \bbS_{n - k}}^{\bbS_n} \calX(k) \times \calY(n - k).
\end{equation}
There is also a composition structure, denoted by $\circ$, on $K_0(\mathsf{Var};\bbS)$. The composition $\calX \circ \calY$ is defined whenever $\calY(0) = \varnothing$, and it is determined by the formula
\begin{equation}\label{eqn:comp_defn}
(\calX \circ \calY)(n) = \coprod_{k \geq 0} (\calX(k) \times \calY^{\boxtimes k}(n))/\bbS_k,
\end{equation}
where $\bbS_k$ acts diagonally. Note that the quotient by $\bbS_k$ exists as a variety since $\bbS_k$ is a finite group.

The following basic combinatorial lemma about compositions of $\bbS$-spaces is essentially immediate from the definition, but it will be a useful identity for a later proof. 

\begin{lem}\label{lem:sigma1plus}
    Suppose $\calX$ and $\calY$ are $\bbS$-spaces such that $\calY(n) = \varnothing$ for $n \leq 1$. Let
    \[\calZ(n):= \begin{cases}
        \Spec \C &\text{if }n = 1\\
        \calY(n) &\text{if }n \neq 1.
    \end{cases}  \]
    Then in $K_0(\mathsf{Var};\bbS)$, we have
    \[ [\calX \circ \calZ] = [\calW] \]
    where in $K_0(\mathsf{Var};\bbS_n)$, we have
    \[ [\calW(n)] = [\calX(n)] + \sum_{j = 0}^{n - 1}\left[\left(\calX(j) \times \coprod_{\substack{{i_1 + \cdots + i_j = n}\\{\text{ at least one }i_\ell > 1}}} \Ind_{\prod_{k = 1}^{j} \bbS_{i_k}}^{\bbS_n} \prod_{k = 1}^{j} \calZ(i_k) \right)/\bbS_j \right].\]
\end{lem}
\begin{proof}
    From the definition of composition (\ref{eqn:comp_defn}) we have the identity
    \[ [\calW(n)] = \sum_{j = 0}^n [(\calX(j) \times \calZ^{\boxtimes j}(n))/\bbS_j] \]
    in the Grothendieck group of $\bbS_n$-varieties. From the definition of $\calZ$, we find that summand corresponding to $j = n$ is exactly $[\calX(n)]$. Now the formula follows by applying the definition of the box product (\ref{eqn:boxtimes}) to each remaining term in the sum: if we have $i_1 + \cdots + i_j = n$ for positive integers $i_\ell$, and $j < n$, then it is automatic that at least one $i_\ell > 1$.
\end{proof}
\subsection{From $\bbS$-spaces to symmetric functions}\label{subsec:sym} Define the graded ring 
\[\Lambda:= \QQ[\![p_1, p_2, \ldots]\!], \]
where $p_i$ has degree $i$. The ring $\Lambda$ is the completion, with respect to degree, of the usual ring of symmetric functions over $\QQ$. Elements of $\Lambda$ are expressions of the form
\[ f = \sum_{d \geq 0} f_d \]
where each $f_d = f_d(p_1, p_2, \ldots)$ is a polynomial which is homogeneous of degree $d$. Given a partition $\mu = (1^{\mu_1}, 2^{\mu_2}, \ldots)$, we define the monomial
\[ p_{\mu} := \prod_{i \geq 1}p_i^{\mu_i} \in \Lambda. \]
For a permutation $\tau \in \bbS_n$, we write $\lambda(\tau)$ for the cycle type of $\tau$, considered as a partition. This lets us define the Frobenius characteristic of a finite-dimensional $\bbS_n$-representation $V$ by the formula
\begin{equation}\label{eqn:frobenius_char_defn}
    \ch_n(V) := \frac{1}{n
!} \sum_{\tau \in \bbS_n} \mathrm{Tr}(\tau|V) \cdot p_{\lambda(\tau)}.
\end{equation} 
With the Frobenius characteristic in hand, we can define the $\bbS_n$-equivariant virtual Poincar\'e polynomial of a variety $X$ with an $\bbS_n$-action:
\[ \mathsf{e}^{\bbS_n}(X) := \sum_{i, j \geq 0} (-1)^i \ch_n(\Gr^{W}_{j} H^i_c(X;\QQ)) \cdot t^j \in \Lambda[t], \]
where
\[\Gr^{W}_{j} H^i_c(X;\QQ) = W_jH^i_c(X;\QQ)/W_{j-1}H^i_c(X;\QQ) \]
are the associated graded pieces of Deligne's weight filtration on the compactly-supported cohomology of $X$. For an $\bbS$-space $\calX$, we define the \textit{Serre characteristic}\footnote{To simplify exposition, we use a less refined version of the Serre characteristic than \cite{GetzlerPandharipande}.} as the generating series
\[ \mathsf{e}(\calX) := \sum_{n \geq 0} \mathsf{e}^{\bbS_n}(\calX(n)) \in \Lambda [\![ t]\!]. \]
Getzler and Pandharipande prove that the Serre characteristic \[\mathsf{e}: K_0(\mathsf{Var};\bbS) \to \Lambda [\![ t]\!] \]
is a ring homomorphism which takes the composition of $\bbS$-spaces (\ref{eqn:comp_defn}) into \textit{plethysm} of symmetric functions \cite[\S5]{GetzlerPandharipande}, defined as follows. First, set
\[F_1\Lambda[\![t]\!] \subset \Lambda [\![t]\!] \]
to be the vector subspace of elements with no constant term; i.e. the terms which do not have bidegree $(0,0)$. Plethysm is the operation \[\Lambda[\![ t ]\!] \times  F_1\Lambda[\![ t ]\!] \to \Lambda[\![ t ]\!] \] denoted by $\circ$, characterized by the following properties:
\begin{enumerate}
\item for any $f \in F_1\Lambda[\![ t ]\!]$, the map $\Lambda[\![ t ]\!] \to \Lambda[\![ t ]\!]$ by $g \mapsto  g \circ f$ is a $\QQ$-algebra homomorphism;
\item for any $n$ and $f \in F_1 \Lambda[\![t]\!]$, we have $p_n \circ f = \psi_n(f)$, where
\[ \psi_n: \Lambda[\![t]\!] \to \Lambda[\![t] \!] \]
is the algebra homomorphism defined by $t \mapsto t^n$ and $p_k \mapsto p_{kn}$ for all $k > 0$.
\end{enumerate}

\subsection{A relation in $K_0(\mathsf{Var};\bbS)$} We now define some important $\bbS$-spaces used in our proof of Theorem \ref{thm:chow_gen_fun}. Let $\Omega^\star\calM$ denote the $\bbS$-space defined by
\[ \Omega^\star\calM(n) = \begin{cases}
\varnothing &\text{if } n \leq 1\\
\Omega^\star\M_{0, n+1} &\text{if }n\geq 2 
\end{cases}\]
(recall that $\Omega^\star \M_{0, n+1}$ was defined in Definition \ref{defn:Omega_star}). We also define an $\bbS$-space $\mathcal{B}$ by
\[ \calB(n) = \begin{cases}
\varnothing &\text{if } n \leq 1\\
\calB_n &\text{if }n\geq 2.
\end{cases}\]
We also set the notation $\sigma_n$ for the class of $\Spec \C$ with trivial $\bbS_n$-action, considered as an $\bbS$-space supported in degree $n$. Theorem \ref{thm:chow_gen_fun} will be deduced from the following identity in $K_0(\mathsf{Var};\bbS)$.
\begin{thm}\label{thm:key_identity}
In $K_0(\mathsf{Var};\bbS)$, the identity
\[ [\calB] \circ (\sigma_1 + [\Omega^\star \M]) = [\calB] + [\C^\star][\calB] - [\Omega^\star\calM]  \]
holds.
\end{thm}
We now explain how Theorem \ref{thm:chow_gen_fun} follows from Theorem \ref{thm:key_identity}. Then the next section will be devoted to the proof of Theorem \ref{thm:key_identity}.
\begin{proof}[Proof of Theorem \ref{thm:chow_gen_fun}]
By \cite[Proposition 13]{Devkota}, we have $\CH^k(\mathcal{B}(n)) \cong H^{2k}(\mathcal{B}(n);\QQ)$ for all $k$, and the space $\calB(n)$ does not have any odd cohomology. Since $\mathcal{B}(n)$ is a smooth projective variety, $H^{2k}(\mathcal{B}(n);\QQ)$ is pure of weight $2k$. Therefore, we can write
\[ \mathsf{e}(\mathcal{B}) = \sum_{n \geq 2} \mathrm{H}_n(t^2) = \mathsf{B}|_{t \mapsto t^2}. \]
Moreover, by Lemma \ref{lem:bundlesoversmooth}, we have 
\[\mathsf{e}(\Omega^\star\M) = \mathsf{e}(\C^\star)\cdot\mathsf{M}|_{t \mapsto t^2}, \]
where $\mathsf{M}$ is defined as in the introduction by 
\[ \mathsf{M} = \sum_{n \geq 2} \sum_{i \geq 0} (-1)^{i} \ch_n(H^i_c(\M_{0, n+1};\QQ)) \cdot t^{i - n + 2}. \]
This holds because $H^i_c(\M_{0, n+1})$ is pure of weight $2(i - n + 2)$, since it is a $(n-2)$-dimensional complex hyperplane arrangement complement. Using $\mathsf{e}(\C^\star) = t^2 - 1$ and making the relevant substitutions into Theorem~\ref{thm:key_identity}, we find that
\[ \mathsf{B}|_{t \mapsto t^2} \circ (h_1 + (t^2 - 1) \mathsf{M}|_{t \mapsto t^2}) = \mathsf{B}|_{t \mapsto t^2} + (t^2 - 1)(\mathsf{B}|_{t \mapsto t^2} - \mathsf{M}|_{t \mapsto t^2}). \]
Rearranging, we obtain
\[ \mathsf{B}|_{t \mapsto t^2} \circ (h_1 + (t^2 - 1) \mathsf{M}|_{t \mapsto t^2}) + (t^2 - 1)\mathsf{M}|_{t \mapsto t^2} = t^2\mathsf{B}|_{t \mapsto t^2}.\]
Adding $h_1$ to both sides, we obtain Theorem \ref{thm:chow_gen_fun}, with the purely aesthetic difference that $t$ has been replaced by $t^2$.
\end{proof}

\section{Proof of Theorem \ref{thm:key_identity}}\label{sec:motivic_proof}
To prove Theorem \ref{thm:key_identity}, we use the stratification of $\calB_n$ by level trees. The basic idea is that the composition $[\calB] \circ (\sigma_1 + [\Omega^\star \M])$ on the left-hand side of Theorem \ref{thm:key_identity} reflects the combinatorial operation of adding either one or zero levels to a level tree, as we will make precise in this section.

Recall that  $\Gamma_{0,n}^{\mathrm{lev},\star}$ denotes the set of isomorphism classes of $(n + 1)$-pointed stable level trees, and $\bbS_n$ naturally acts on this set by permuting the markings $\{1, \ldots, n\}$.
\begin{defn}
  We use the notation $[\bT, \ell]$ for the $\bbS_n$-orbit of a level tree $(\bT, \ell)$. We define $\mathcal{B}_{[\bT,\ell]}$ to be an $\bbS$-space supported in degree $n$, where
\[ \mathcal{B}_{[\bT,\ell]}(n) \subset \calB(n) \]
is the $\bbS_n$-invariant subspace consisting of curves with level tree in the orbit $[\bT, \ell]$. 
\end{defn}
Now we define another $\bbS$-space $\Omega^\star \M_{aug}$.
\begin{defn}\label{defn:omega_star_aug}
Define $\Omega^\star\calM_{aug}$ by
\[ \Omega^\star\calM_{aug}(n) = \begin{cases}
    \Spec \C &\text{if }n=1 \\
    \Omega^\star\M(n) &\text{otherwise.}
\end{cases}   \]
\end{defn}
The reason for defining $\Omega^\star \M_{aug}$ is that in $K_0(\mathsf{Var};\bbS)$, we have the identity
\[ [\Omega^\star \M_{aug}] = \sigma_1 + [\Omega^\star \M].  \]

\subsection{Boxed powers of $\Omega^\star \M_{aug}$}\label{subsec:boxed_powers} Towards understanding the composition $[\calB \circ \Omega^\star \M_{aug}]$, we will first state how to think of each boxed power $\Omega^\star \M_{aug}^{\boxtimes r}$ as an auxiliary moduli space for each integer $r \geq 1$. Indeed, unwinding the definition of the box product (\ref{eqn:boxtimes}), we find that for each $n \geq 1$, we have that $\Omega^\star \M_{aug}^{\boxtimes r}(n)$ is the moduli space of ordered tuples
\[ (C_1, \ldots, C_r, f) \]
where
\begin{enumerate}
    \item $C_j$ is either 
    \begin{itemize}
        \item a multiscale differential with distinguished marking $q_j$, as well as $i_j$ additional non-distinguished markings $p_{k_1}, \ldots, p_{k_{i_j}}$, or
        \item a single point $p_j$ which we may treat as a non-distinguished marking.
    \end{itemize}
    \item If we let $M$ denote the set $\{p \mid p \in C_j \, \text{a non-distinguished marking  for some }j \}$, then 
    \[f: M \to \{1, \ldots, n\} \]
    is a bijection.
\end{enumerate} 
The $\bbS_n$-action on $\Omega^\star \M_{aug}^{\boxtimes r}(n)$ is induced by permutations of $\{1, \ldots, n\}$. Note that there is a natural $\bbS_r$-action on $\Omega^\star \M^{\boxtimes r}_{aug}$ by re-ordering the multiscale differentials. This $\bbS_r$-action appears in the definition of the composition operation (\ref{eqn:comp_defn}). 

\subsection{Pruning extremal vertices}
We now turn towards a combinatorial interpretation of the composition $[\calB \circ \Omega^\star \M_{aug}]$. Given a level tree $(\bT, \ell)$, we write $V^{\mathrm{ext}}(\bT)$ for the set of vertices of $\bT$ of the highest level, and set
\[r(\bT, \ell) := |V^{\mathrm{ext}}(\bT)|. \]
\begin{defn}
    For a level tree $(\mathbf{T}, \ell) \in \Gamma_{0, n}^{\mathrm{lev}, \star}$, we define the \textit{pruning} of $(\T, \ell)$ as a level tree \[(\T^{\mathrm{pr}}, \ell^{\mathrm{pr}}) \in \Gamma_{0, r(\T, \ell)}^{\mathrm{lev}, \star} \]by deleting the $r(\T, \ell)$ vertices in $V^{\mathrm{ext}}(\T)$ from $\T$ and replacing them with markings. The ordering of the markings on $(\T^{\mathrm{pr}}, \ell^{\mathrm{pr}})$ is the unique one which is compatible with the lexicographic ordering of the subsets of $\{1, \ldots, n\}$. The level structure $\ell^{\mathrm{pr}}$ is the restriction of the level structure~$\ell$.
\end{defn}
 Note that the pruning always has exactly one less level than the original level tree. See Figure~\ref{fig:pruning} for an example of a pruning.

\begin{figure}
    \centering
    \includegraphics[scale=1]{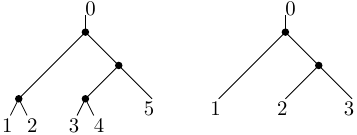}
    \caption{A level tree with $3$ levels on the left, and its pruning on the right. Here, a marking on a vertex $v$ is visualized by a labelled half-edge emanating from $v$.}
    \label{fig:pruning}
\end{figure} 

Now we define \[U([\T, \ell]) \subseteq \coprod_{k \geq 2} \Gamma_{0, k}^{\mathrm{lev, \star}} \]
to be the set of all level trees whose pruning is in the same $\bbS_n$-orbit as $(\T, \ell)$. The following lemma encodes the recursive structure obtained by pruning, at the level of $\bbS$-spaces.

\begin{lem}\label{lem:pruningplethysm}
In the Grothendieck ring of $\bbS$-spaces, the identity
\[[\calB_{[\bT, \ell]}] \circ (\sigma_1 + [\Omega^\star \M]) = [\calB_{[\T, \ell]}] + [\C^\star]\cdot \sum_{[\bT', \ell'] \in U([\bT, \ell])}[\calB_{[\bT', \ell']}]\]
holds.
\end{lem}
\begin{proof}
The left-hand side is the class in $K_0(\mathsf{Var};\bbS)$ of
\[ \calB_{[\T, \ell]} \circ \Omega^\star\calM_{aug}, \]
where $\Omega^\star \M_{aug}$ is the $\bbS$-space defined in Definition \ref{defn:omega_star_aug}. Then Lemma \ref{lem:sigma1plus} implies that
\begin{equation}\label{eqn:comb_lemma_part1}
[\calB_{[\T, \ell]} \circ \Omega^\star \M_{aug}] = [\calB_{[\T, \ell]}] + \left[\calX \right]
\end{equation}
where $\calX$ is the $\bbS$-space such that
\begin{equation}\label{eqn:comb_lemma_Xdefn}
    \calX(n) = \left(\calB_{[\bT, \ell]}\left(r(\T, \ell)\right) \times \coprod_{\substack{{i_1 + \cdots + i_{r(\T, \ell)} = n}\\{\text{at least one }i_j > 1}}} \Ind_{\prod_{j} \bbS_{i_j}}^{\bbS_n} \prod_{j = 1}^{r(\T, \ell)}\Omega^\star\calM(i_j)\right)/\bbS_{r(\T, \ell)}
\end{equation}
The disjoint union appearing in (\ref{eqn:comb_lemma_Xdefn}) is a subspace of the moduli space $\Omega^\star \M_{aug}^{\boxtimes r(\T, \ell)}(n)$ described in \S\ref{subsec:boxed_powers}: it is exactly the subspace where at least one of the $C_1, \ldots, C_{r(\T, \ell)}$ is not a point (that is, we have at least one honest differential). Now generalizing Lemma \ref{lem:stratum_factorization}, we find that 
\begin{equation}\label{eqn:comb_lemma_Xquotient}
\calX(n)/\C^\star = \coprod_{[\T', \ell'] \in U([\T, \ell])} \calB_{[\T', \ell']}(n), 
\end{equation}
where the disjoint union is finite, because only finitely many $[\T', \ell'] \in U([\T, \ell])$ can have $\calB_{[\T', \ell']}(n) \neq \varnothing$. Using the interpretation in \S\ref{subsec:boxed_powers},
the right-hand side of (\ref{eqn:comb_lemma_Xdefn}) should be thought of as adding a single level to a multiscale differential in $\calB_{[\T, \ell]}$, and the (free) $\C^\star$-quotient on the left-hand side of (\ref{eqn:comb_lemma_Xquotient}) corresponds to rescaling the differentials at the new level. Considering all $n$ at once, we find
\begin{equation}\label{eqn:comb_lemma_Xclass}
    [\calX] = [\C^\star]\cdot\sum_{[\T', \ell'] \in U([\T, \ell])} [\calB_{[\T', \ell']}
    \end{equation}
in $K_0(\mathsf{Var}; \bbS)$.  Substituting (\ref{eqn:comb_lemma_Xclass}) into (\ref{eqn:comb_lemma_part1}), we obtain the desired statement.
\end{proof}

Using Lemma \ref{lem:pruningplethysm}, we can complete the proof of Theorem \ref{thm:key_identity}.
\begin{proof}[Proof of Theorem \ref{thm:key_identity}]
Let $\calB^{(k)} \subset \calB$ denote the $\bbS$-subspace parameterizing stable curves whose level tree has length $k$ (recall that the length of a level tree is the number of levels). Then by summing the equality of Lemma \ref{lem:pruningplethysm} over all $[\T, \ell]$ with length equal to $k$, we obtain
\[ [\calB^{(k)}] \circ (\sigma_1 + [\Omega^\star \M]) = [\calB^{(k)}] +  [\C^\star][\calB^{(k+1)}] \]
for $k \geq 1$. Now we sum over all $k \geq 1$ to find that
\[ [\calB] \circ (\sigma_1 + [\Omega^\star \M]) =[\calB] + [\C^\star]([\calB] - [\calB^{(1)}]).  \]
Theorem \ref{thm:key_identity} now follows from Lemma \ref{lem:bundlesoversmooth}, which implies that $[\C^\star][\calB^{(1)}] = [\Omega^\star\M]$.
\end{proof}

\section{Numerical applications}\label{sec:numerical}
We now show how to deduce Corollary \ref{cor:numerics} from Theorem \ref{thm:chow_gen_fun}. 
\begin{proof}[Proof of Corollary \ref{cor:numerics}]
The homomorphism
\[ \mathrm{rk}:\Lambda[\![ t]\!] \to \QQ[\![x, t]\!], \]
taking $p_1 \to x$ and $p_n \to 0$ for $n > 1$ satisfies
\[ \mathrm{rk}(\ch_n(V)) = \dim_{\QQ}(V)\frac{x^n}{n!} , \]
as can be seen directly from the definition of $\ch_n(V)$ in (\ref{eqn:frobenius_char_defn}). Moreover, one can verify that if $f, g \in \Lambda[\![ t]\!]$ and $g$ is in the ideal $(p_1, p_2, \ldots)$, then $\mathrm{rk}(f \circ g) = \mathrm{rk}(f) \circ \mathrm{rk}(g)$, where on the right $\circ$ denotes composition of power series in the variable $x$.
Applying this to Theorem \ref{thm:chow_gen_fun}, we obtain
\begin{equation}\label{eq:num_gen_fun}
(x + \mathrm{rk}(\mathsf{B})) \circ (x + (t - 1) \mathrm{rk}(\mathsf{M})) = x + t \cdot \mathrm{rk}(\mathsf{B}).
\end{equation}
Now define polynomials $\omega_n(t)$ inductively by setting $\omega_1(t) := 1$ and
\[ \omega_n(t) := (t - n + 2) \cdot \omega_{n - 1}(t) \]
for $n \geq 2$.
The coefficient of $x^n$ in $\mathrm{rk}(\mathsf{M})$ is
\[ \frac{1}{n!} \cdot\frac{\omega_n(t-1)}{t - 1}; \] this equality may be established e.g. by using the isomorphism \[\M_{0, n + 1} \cong \mathrm{Conf}_{n - 2}(\P^1 \smallsetminus \{0, 1, \infty\}) \]
and counting the $\F_q$-points of $\Conf_{n - 2}(\P^1 \smallsetminus \{0, 1, \infty\})$. Therefore
\[ x + (t-1)\mathrm{rk}(\mathsf{M}) = x + \sum_{n\geq 2} \frac{\omega_n(t-1)}{n!}\cdot x^n. \]
Since
\[x + \mathrm{rk}(\mathsf{B}) =  \sum_{n \geq 2} \mathrm{H}^{\mathrm{num}}_n(t) \cdot \frac{x^n}{n!}, \]
we obtain from (\ref{eq:num_gen_fun}) and Fa\`a di Bruno's formula the identity 
\begin{equation}\label{eqn:faadibruno}
t \cdot \mathrm{H}^{\mathrm{num}}_n(t) = \sum_{k = 1}^{n} \mathrm{Bell}_{n, k}(\omega_1(t-1), \ldots, \omega_{n - k + 1}(t-1)) \cdot \mathrm{H}_{k}^{\mathrm{num}}(t)
\end{equation}
for $n \geq 2$. Above, $\mathrm{Bell}_{n, k}(x_1, \ldots, x_{n - k + 1})$ is the \textit{partial exponential Bell polynomial}, defined by the formula
\begin{equation}\label{eqn:bell_defn}
\mathrm{Bell}_{n, k}(x_1, \ldots, x_{n - k + 1}) = n!\sum_{\substack{{\lambda \vdash n}\\{\ell(\lambda) = k}}} \prod_{i = 1}^{n - k + 1} \frac{x_i^{\lambda_i}}{i!^{\lambda_i} \lambda_i!}
\end{equation}
where for a partition $(1^{\lambda_1}, 2^{\lambda_2},\ldots)$ we write $\ell(\lambda) = \sum_i \lambda_i$; see \cite[\S3.3,\S3.4]{Comtet}. Subtracting $\mathrm{H}_n^{\mathrm{num}}(t)$ from both sides of (\ref{eqn:faadibruno}) we see that
\begin{equation}\label{eq:num_recursion}
(t - 1) \cdot \mathrm{H}^{\mathrm{num}}_n(t) = \sum_{k = 1}^{n - 1} \mathrm{Bell}_{n, k}(\omega_1(t-1), \ldots, \omega_{n - k + 1}(t-1)) \cdot \mathrm{H}_{k}^{\mathrm{num}}(t).
\end{equation}

Now the identity
\[ t^k \cdot\mathrm{Bell}_{n, k}(\omega_1(t-1), \ldots, \omega_{n - k + 1}(t-1)) = \sum_{j = k}^{n} s(n, j) S(j, k) t^{j} \]
from \cite[Theorem 2.1]{bellpoly} yields 
\[ (t-1)\cdot \mathrm{H}_n^{\mathrm{num}}(t) = \sum_{k = 1}^{n - 1} \mathrm{H}_{k}^{\mathrm{num}}(t)\cdot \sum_{j = k}^{n} s(n, j)S(j, k)t^{j - k}, \]
which is the first part of Corollary \ref{cor:numerics}. To obtain the second part of Corollary \ref{cor:numerics}, we divide both sides of (\ref{eq:num_recursion}) by $(t - 1)$, and then take the limit of both sides as $t \to 1$. We find
\begin{align*}\label{eqn:limitsofbell}
     \frac{\mathrm{Bell}_{n, k}(\omega_1(t-1), \ldots, \omega_{n - k + 1}(t-1))}{(t-1)} &=   n!\sum_{\substack{{\lambda \vdash n}\\{\ell(\lambda) = k}}} \frac{1}{(t - 1)}\prod_{i = 1}^{n - k + 1} \frac{\omega_i(t-1)^{\lambda_i}}{i!^{\lambda_i} \lambda_i!}\\&= n!\sum_{\substack{{\lambda \vdash n}\\{\ell(\lambda) = k}}} \frac{1}{\lambda_1 !} \cdot \frac{1}{(t - 1)}\prod_{i = 2}^{n - k + 1} \frac{(t-1)^{\lambda_i} \cdots (t - i + 1)^{\lambda_i}}{i!^{\lambda_i} \lambda_i!}.
\end{align*}
Direct inspection of the above expression shows that the summand corresponding to $\lambda \vdash n$ has nonzero limit as $t \to 1$ if and only if $\lambda = (1^{\lambda_1}, j^1)$ for some $j$. Moreover, since $\ell(\lambda) = k$, we must have $\lambda_ 1 + 1 = k$, so $\lambda_1 = k - 1$, and $j = n -k + 1$. Therefore we conclude that
\[\lim_{t \to 1} \frac{\mathrm{Bell}_{n, k}(\omega_1(t-1), \ldots, \omega_{n - k + 1}(t-1))}{(t-1)} = n! \cdot \frac{1}{(k - 1)!} \frac{(-1)^{n - k - 1}(n - k - 1)!)}{(n - k + 1)!}. \]
This expression, together with (\ref{eq:num_recursion}), implies the second part of Corollary \ref{cor:numerics}.
\end{proof}

\begin{rem}\label{rem:general_recursion}
	The recursion~\eqref{eq:num_recursion} for $\mathrm{H}_n^{\mathrm{num}}(t)$ also follows from the recursion in~\cite{JKUmotivic,FMSV24}: For a general matroid $M$ with Chow polynomial $\mathrm{H}^{\mathrm{num}}_M(t)$ we have
	\[
		 (t-1)\cdot\mathrm{H}^{\mathrm{num}}_M(t) = \sum_{\substack{F\in L(M)\\ F\neq \emptyset}}\chi_{M\mid_F}(t)\cdot\mathrm{H}^{\mathrm{num}}_{M/F}(t),
	\]
	where
    \begin{itemize}
        \item $L(M)$ is the lattice of flats of $M$,
        \item $\chi_M(t)$ is the characteristic polynomial of $M$, and
        \item $M|_F$ and $M/F$ are the restriction of $M$ to $F$ and the contraction of $M$ by $F$, respectively.
    \end{itemize}
	Specializing to the case of the braid matroid, we know that its lattice of flats coincides with the partition lattice $\Pi_n$, and its characteristic polynomial is $\chi_n(t)=\sum_{j=1}^n s(n,j)t^{j-1}=\omega_n(t-1)$.
	The interval in the partition lattice below a partition $\pi=(B_1,\dots,B_k)\in \Pi_n$, i.e., the restriction $B_n|_\pi$, is the product of smaller partition lattices on the blocks $B_j$.
	Similarly, the interval in this lattice above the partition  $\pi=(B_1,\dots,B_k)\in \Pi_n$, i.e., the contraction $B_n/\pi$, is the partition lattice on $k$ elements. Thus we obtain the recursion
	\begin{align*}
		(t-1)\cdot \mathrm{H}_n^{\mathrm{num}}(t) = \sum_{k=1}^{n-1}\mathrm{H}_k^{\mathrm{num}}(t)\sum_{(B_1,\dots,B_k)\in \Pi_n} \prod_{j=1}^{k}\omega_{|B_j|}(t-1).
	\end{align*}
	The partial Bell polynomial $B_{n,k}$ is a $(n-k+1)$-variate polynomial encoding the partitions of~$[n]$ into $k$ blocks where the variable $x_i$ indicates the presence of a block of size $i$.
	Thus, the last equation is equivalent to the recursion~\eqref{eq:num_recursion} described above.
\end{rem}

\section{Relative stable maps to $\P^1$}\label{sec:relmaps}
Fix now a tuple of integers \[\boldsymbol{x} = (x_1, \ldots, x_n) \in (\Z\smallsetminus \{0\})^n,\] such that
\[ \sum_{i = 1}^n x_i = 0.\]
Let $\M(\boldsymbol{x})$ denote the moduli space of maps \[f:(\P^1, q_1, \ldots, q_n) \to (\P^1, 0, \infty),\] where
\begin{itemize}
\item $q_i \in \P^1$ are distinct marked points, 
\item $f^{-1}(0) = \{q_i \mid x_i > 0\}$,
\item $f^{-1}(\infty) = \{q_i \mid x_i < 0\}$, and
\item $f$ ramifies with degree $|x_i|$ at $q_i$. 
\end{itemize}
Two such pointed maps are considered equivalent if they fit into a commuting square. In particular, two maps are equivalent when they differ by the $\C^\star$-action on $\P^1$; this target is sometimes called a \textit{rubber} $\P^1$. There is an isomorphism $\M(\boldsymbol{x}) \cong \M_{0, n}$, since a morphism $\P^1 \to \P^1$ is determined up to the $\C^\star$-action on the target by its zeros and poles.  Relative Gromov--Witten theory affords a modular compactification \[\M(\boldsymbol{x}) \hookrightarrow \Mbar(\boldsymbol{x}),\]
where $\Mbar(\boldsymbol{x})$ parameterizes so-called relative stable maps to a rubber $\P^1$; see e.g. \cite[\S2.2]{CMR} or \cite[\S 0.2.3]{FaberPandharipande} for a precise definition. In \cite{CMR}, Cavalieri, Markwig, and Ranganathan exhibit $\Mbar(\boldsymbol{x})$ as a tropical compactification of $\M_{0, n}$. Their work implies that there is a fiber square
\begin{equation}
    \begin{tikzcd}\label{eqn:toric_pullback}
    &\Mbar(\boldsymbol{x}) \arrow[d] \arrow[r] &X(\Sigma_{n, \boldsymbol{x}}) \arrow[d]\\
    &\Mbar_{0, n} \arrow[r] &X(\Sigma_{n})
\end{tikzcd}
\end{equation}
where 
\begin{itemize}
    \item $\Sigma_n$ is the \textit{Bergman fan} of the braid matroid $B_n$ with respect to the minimal building set,
    \item the fan $\Sigma_{n, \boldsymbol{x}}$ is a certain subdivision of $\Sigma_n$ determined by the vector $\boldsymbol{x}$ (this fan is denoted by $\Delta_{\boldsymbol{x}}^{\mathrm{rub}}$ in \cite{CMR}),
    \item $X(\Sigma)$ denotes the toric variety of a fan $\Sigma$, and
    \item the map $\Mbar_{0, n} \to X(\Sigma_n)$ is the tropical compactification established in \cite{Tevelev} and \cite{GibneyMaclagan}. 
\end{itemize}
The fan $\Sigma_n$ is isomorphic as a cone complex to the moduli space of $n$-pointed stable tropical curves of genus zero, and the subdivision $\Sigma_{n, \boldsymbol{x}} \to \Sigma_n$ has a pleasing modular interpretation in terms of tropical curves \cite[Proposition 16]{CMR}. We now specialize to a particular family of ramification vectors $\boldsymbol{x}$ for our theorem on relative stable maps.
\begin{defn}\label{defn:Rn}
    Assume $n \geq 2$. We define $\calR_n$ to be the moduli space $\Mbar(\boldsymbol{x})$ when \[\boldsymbol{x} = (n, \underbrace{-1, \ldots, -1}_{n}).\] 
\end{defn}
The topological Euler characteristic of $\calR_n$ was determined by the first author in \cite{KannanP1}. Via Corollary \ref{cor:numerics}, the following theorem determines the dimensions of the Chow groups of $\calR_n$.
\begin{thm}\label{thm:relative_stable_maps}
For each $n \geq 2$ and $k \geq 0$, there is an equality of dimensions
\[ \dim_{\QQ} \CH^k(\calR_n) = \dim_{\QQ} \CH^k(B_n). \]
\end{thm}
\begin{proof}
Translating the definition of $n$-marked combinatorial rubber maps in \cite[Definition 2.5]{KannanP1} to the present setting, the moduli space $\calR_n$ is stratified by level trees, just as in the case of $\calB_n$. Via \cite[Equation 2.14]{KannanP1} the coarse moduli space corresponding to a level tree factors as a product as in Lemma \ref{lem:stratum_factorization}. In particular, this leads to an equality $ [\calR_n] = [\calB_n] $
of the classes of the two moduli spaces in the Grothendieck ring of varieties $K_0(\mathsf{Var})$. This equality can also be proved by analyzing the torus fibers of the vertical maps in (\ref{eqn:toric_pullback}).
Since both spaces are smooth and proper Deligne--Mumford stacks, it follows that
\[ \dim_{\QQ} H^{2k}(\calR_n;\QQ) = \dim_{\QQ}H^{2k}(\calB_n;\QQ) \]
for all $k \geq 0$. Both spaces have the so-called \textit{Chow--K\"unneth generation property} (\cite[Definition 3.1]{CL}) by \cite[Lemma 3.4]{CL}. This property implies that the cycle class map from the Chow to the cohomology ring is an isomorphism for both spaces, and the statement of the theorem follows.
\end{proof}

\begin{rem}
    By \cite[Theorem B]{KannanP1}, Theorem \ref{thm:relative_stable_maps} also holds if $\calR_n$ is replaced by any moduli space of relative stable maps $\Mbar(\boldsymbol{x})$ where $\boldsymbol{x} \in (\Z\smallsetminus \{0\})^{n + 1}$ has exactly one positive entry.
 \end{rem}

\begin{rem}
The equality of dimensions in Theorem \ref{thm:relative_stable_maps} cannot be upgraded to a space-level isomorphism: the moduli space $\calR_n$ is only smooth when viewed as a Deligne--Mumford stack, and its coarse moduli space has finite quotient singularities. On the other hand, $\calB_n$ is a smooth projective variety. We do expect that Theorem \ref{thm:relative_stable_maps} can be upgraded to an equality of characters of $\bbS_n$-representations. 
It is also natural to ask whether the equality of Chow groups in Theorem~\ref{thm:relative_stable_maps} can be upgraded to an isomorphism of Chow rings. We expect that the two rings are not isomorphic.
\end{rem}

 \printbibliography

\end{document}